\documentclass[a4paper,11pt]{amsart}
\usepackage[utf8]{inputenc}
\usepackage{amsfonts,amssymb}
\usepackage{graphicx,tikz}
\usepackage{tikz-cd} 
\usepackage{float}
\usepackage{amsmath, amsthm, amsfonts}
\usepackage{geometry}
\usepackage{hyperref}
\usepackage{enumitem}  
\usepackage[capitalise]{cleveref}
\usepackage{comment}
\usepackage{enumitem}
\usepackage{amsrefs}
\usepackage{nicefrac}
\usepackage[official]{eurosym}
\usepackage{graphicx}
\usepackage{nomencl}
\usepackage{amsfonts}
\usepackage{hyperref}
\usepackage{amssymb}
\usepackage{fancyhdr}
\usepackage{amscd}
\usepackage[english]{babel}

\newtheorem*{theorem*}{Theorem}
\newtheorem{thm}{Theorem}
\newtheorem{pr}[thm]{Proposition}
\newtheorem{cor}[thm]{Corollary}
\newtheorem{lem}[thm]{Lemma}
\newtheorem{theorem}{Theorem}
\newtheorem{corollary}[theorem]{Corollary}

\theoremstyle{definition}

\newtheorem{rem}[thm]{\scshape{Remark}}
\newtheorem{example}[thm]{\scshape{Example}}

\def\r{\rho}
\def\a{\alpha}

\keywords{Group element orders, Frobenius groups, Sylow tower}

\begin{document}

\title{Upper bounds for the product of element orders of finite groups}

\author[E. Di Domenico]{Elena Di Domenico}
\address{Elena Di Domenico: Department of Mathematics, University of Trento, 38123, Povo (TN)  Italy; Matematika Saila, University of the Basque Country UPV/EHU, 48080 Bilbao Spain}
 \email{elena.didomenico@yahoo.it}
 
\author[C. Monetta]{Carmine Monetta}
\address{Carmine Monetta: Dipartimento di Matematica, Universit\`a di Salerno, Italy}
\email{cmonetta@unisa.it}

\author[M. Noce]{Marialaura Noce}
\address{Marialaura Noce: Dipartimento di Matematica, Universit\`a di Salerno, Italy }
\email{mnoce@unisa.it}

\thanks{The authors are partially supported by the ``National Group for Algebraic and Geometric Structures, and their Applications'' (GNSAGA - INdAM). The first and the last author are supported by the Spanish Government grant MTM2017-86802-P, partly with FEDER funds, and by the Basque Government grant IT974-16. The first author  acknowledges support from the Department of Mathematics of the University of Trento. The third author acknowledges support from  EPSRC, grant number
1652316.}

\maketitle

\begin{abstract}
Let $G$ be a finite group of order $n$, and denote by $\rho(G)$ the product of element orders of $G$. The aim of this work is to provide some upper bounds for $\rho(G)$ depending only on $n$ and on its least prime divisor, when $G$ belongs to some classes of non-cyclic groups.
\end{abstract}

\section{Introduction}
\noindent It is well-known that the behaviour of element orders strongly affects the structure of a periodic group. For instance, in the finite case, a group can be characterized by looking at its element orders, as the group \mbox{PSL}$(2,q)$ for $q \neq 9$, see \cite{BW}.
%It is well known that the study of structural information of a finite group by looking at its element orders is a classical topic in Group Theory.
Therefore, it is natural to study the set $\omega(G)=\{o(x) \ | \ x \in G \}$ of a given periodic group, where $o(x)$ denotes the order of $x \in G$. To give an example, if $\omega(G)\subseteq\{1,2,3,4\}$, the group $G$ is locally finite \cite{sanov}. For a survey on this topic see \cite{HLM1}.
%The research of structural information about a periodic group by looking at its element orders is by now a classical topic in Group Theory. For instance, given a periodic group $G$, one can study the set $\omega(G)=\{o(x) \ | \ x \in G \}$, where $o(x)$ denotes the order of $x \in G$: for a survey on this topic see for example \cite{HLM1}.
Furthermore, one can impose some arithmetic conditions on element orders of a subset $S$ of a group $G$ to obtain information about the subgroup generated by $S$, see for instance \cites{BM,BMS, BS, MT}. Moreover, some recent criteria for solvability, nilpotency and other properties of finite groups $G$, based either on the orders of the elements of $G$ or on the orders of the subgroups of $G$ have been described in \cite{HLM5}.

Another direction is to consider functions depending on element orders. In  \cite{AJI}, H. Amiri et al. introduced the function $\psi(G)$, which denotes the sum of element orders of a finite group $G$. They proved that if $G$ is a non-cyclic group of order $n$, then $\psi(G) < \psi(C_n)$, where $C_n$ denotes the cyclic group of order $n$ (see also \cite{AJ} and \cite{Jafarian}). Recently, this last result has been improved by M. Herzog et al., who proved in \cite{HLM} that if $G$ is a finite non-cyclic group then $\psi(G) \leq \frac{7}{11}\psi(C_n)$, by showing also that this bound is the best possible; see also \cites{HLM2, HLM3, HLM4}.

In the present work, we will denote by $\r(G)$ the product of element orders of a finite group $G$.  Some finite groups can be recognized by the product of its element orders, like the groups $PSL(2,7)$ and $PSL(2,11)$, while the groups $PSL(2,5)$ and $PSL(2,13)$ are uniquely determined by their orders and the product of element orders, see \cite{AK2}. However, one can easily see that in general the knowledge of $\r(G)$ is not sufficient to recognize the group $G$, even when its order is known. Indeed, if we denote by $S_4$ the symmetric group of degree $4$, and by $D_{12}$ the dihedral group of order $12$, then $S_4$ and $C_2 \times D_{12}$ have the same order and $\r(S_4)=\r(C_2 \times D_{12})$. This implies the necessity to found under which conditions $\rho(G)$ affects the structure of the group $G$. 

The function $\rho$ was studied  by M. Garonzi and M. Patassini in \cite{GP}, where they proved that $\r(G) \leq \r(C_n)$ for every finite group $G$ of order $n$, and that $\r(G) = \r(C_n)$ if and only if $G\simeq C_n$. Later M. T\u{a}rn\u{a}uceanu in \cite{tarnauceanu} studied $\rho(G)$ when $G$ is a finite abelian group $G$, showing that two finite abelian groups of the same order are isomorphic if and only if they have the same product of element orders. 

The aim of this paper is to provide more information about the function $\r$. We consider some classes of non-cyclic finite groups, and we determine upper bounds for the product of element orders depending only on the order of the group and on the smallest prime dividing it. 

Our first result reads as follows.

\begin{theorem}\label{mainthm}
Let $G$ be a non-cyclic supersoluble group of order $n$. If either $G$ is nilpotent or $G$ is not metacyclic, then

\[
%\r(G) \leq  \frac{1}{q^{\frac{n}{q} (q-1)}}\r(C_{n}).
\r(G) \leq q^{-\frac{n}{q} (q-1)}\r(C_{n}).
\]

where $q$ is the smallest prime dividing $n$.
\end{theorem}

We recall that a group $G$ is said to admit a {\it Sylow tower} if there exists a normal series
\[
1 = G_0 \leq G_1 \leq \cdots \leq G_n=G
\]
such that each $G_{i+1}/G_i$ is isomorphic to a Sylow subgroup of $G$ for every $i \in \{0, \ldots, n-1\}$.

As a consequence of Theorem~\ref{mainthm} we obtain the following bound for groups admitting a Sylow tower.

\begin{corollary}
Let $G$ be a non-cyclic group of order $n$ admitting a Sylow tower. 
If $q$ is the smallest prime dividing $n$, then

\[
\r(G) \leq  q^{-q} \rho(C_n).
\]

\end{corollary}

Going further, we show that such a bound holds for other class of groups. Indeed, if $G$ is a non-cyclic group of order $n$ and $q$ is the smallest prime dividing $n$, we show that 
\[
\rho(G) \leq q^{-q} \rho(C_n),
\]
when either $n=p^{\alpha}q^{\beta}$, where $p > q$ are primes (Theorem \ref{thm:pq}), or $G$ is a Frobenius group (Proposition \ref{frobenius}).

%%%%%%%%%%%%%%%%%%%% SECTION

\section{Preliminary results}

In this section we recall some results concerning the function $\r$, then we compute a bound for the product of element orders of a group with a normal Sylow $p$-subgroup, and finally we show the main result in the case of non-cyclic nilpotent groups.

The following results from \cite{tarnauceanu} will be useful in the next.

\begin{lem}\cite[Prop.~1.1] {tarnauceanu}\label{rhocoprime}
Let $n \geq 1$, and let $H_1,\ldots,H_n$ be finite groups with pairwise coprime orders. Then 
\[
 \r\left( H_1 \times \cdots \times H_n \right)=\prod_{i=1}^{n}\r(H_i)^{\prod_{j\neq i}|H_j|}.
\]
\end{lem}

As a consequence, and by using \cite[Theorem~1.1] {tarnauceanu}, one can readily see that the following corollary holds.

\begin{cor}\cite[Ex.~1.1] {tarnauceanu}\label{summ} 
Let $C_n$ be a cyclic group of order $n$.
\begin{enumerate}[label=(\roman*)]
    \item If $n=p^{\a}$ for some prime $p$, then
    \[
    \r(C_n)= p^{\displaystyle \frac{\alpha p^{\alpha+1}-(\alpha+1)p^\alpha+1}{p-1}}.
    \]
    \item If $n=p_1^{\alpha_1}p_2^{\alpha_2}\cdots p_s^{\alpha_s}$, where $p_i$'s are distinct primes and $\a_i$'s are positive integers, then
    \[
    \r(C_n)=\prod_{i=1}^{s}p_i^{\left(\displaystyle \frac{\alpha_i p_i^{\alpha_i+1}-(\alpha_i+1)p_i^{\alpha_i}+1}{p_i-1}\right){\displaystyle \frac{n}{p_i^{\a_i}}}}.
    \]
\end{enumerate}
\end{cor}

The following remarks will be useful in the next.

\begin{rem}\label{rem:p}
Let $p$ be a prime and $\alpha \geq 1$. Then
\[
(p^{\alpha -1})^{p^{\alpha}} \leq \rho(C_{p^{\alpha}}) p^{-1}.
\]

If, additionally, $p$ is odd and $\alpha \geq 2$, we have
\[
(p^{\alpha -1})^{p^{\alpha}} \leq \rho(C_{p^{\alpha}})p^{-p}.
\]
\end{rem}
\begin{comment}
\begin{proof}
By Corollary~\ref{summ}(i), we only need to check that
\[
(p-1)(\alpha -1)p^{\alpha} \leq \alpha p^{\alpha+1} - (\alpha +1)p^{\alpha} +1 - p +1,
\]
which easily follows.

Now let $p$ be odd, and $\alpha > 1$. Again, by Corollary~\ref{summ}(i), we only need to check that
\[
(p - 1)(\alpha - 1)p^{\alpha} \leq \alpha p^{\alpha+1} - (\alpha +1)p^{\alpha} +1 - p^2+p.
\]
However, this is equivalent to
%\[
%p^{\alpha+1} -2p^{\alpha}- p^2+ p + 1 \geq 0,
%\]
%that is
\[
p^{\alpha}(p -2)  \geq  p^2 - p - 1,
\]
which is true as $p >2$ and $\alpha >1$.
\end{proof}
\end{comment}

\begin{rem}\label{rem: q<p}
Let $p \geq q$ be positive integers. Then
\[
p^{-(p-1)/p} \leq q^{-(q-1)/q}
\]
\end{rem}
\begin{comment}
\begin{proof}
We observe that the thesis is equivalent to show that 
\[
p(q-1)\log_p(q)\leq q(p-1).
\]
Since $q\leq p$ then $p(q-1)\leq q(p-1)$. Thus we have
\[
p(q-1)\log_p(q)\leq p(q-1)\leq q(p-1).
\]
This completes the proof.
\end{proof}
\end{comment}

%An immediate consequence of Remark~\ref{rem:p} is the following.
\begin{comment}
For a non-cyclic $p$-group we have the following result.
\begin{pr}\label{pr: nilpotent}
(Unire a Proposizione 8)Let $G$ be a non-cyclic group of order $q^{\alpha}$, for $q$ a prime. Then
\[
\r(G) \leq  \frac{1}{q^{\frac{|G|}{q} (q-1)}}\r(C_{|G|}).
\]
\end{pr}
\begin{proof}
As $G$ is non-cyclic, $o(x) \leq q^{\alpha -1}$ for every $x \in G$. Let $M$ be a maximal subgroup of $G$. Applying \cite{GP}*{Theorem 3} we have
\begin{align*}
    \r(G)\leq \r(M)\prod_{x\in G\setminus M}o(x)\leq \r(C_{q^{\alpha-1}})(q^{\alpha-1})^{q^{\alpha}-q^{\alpha-1}}.
\end{align*}
Since $\r(C_{q^{\alpha}})=\r(C_{q^{\alpha-1}})q^{\alpha(q^{\alpha}-q^{\alpha-1})}$, it follows that
\begin{align*}
    \r(G)\leq \r(C_{q^{\alpha}})\frac{1}{q^{\alpha}-q^{\alpha-1}}=\frac{1}{q^{\frac{|G|}{q} (q-1)}}\r(C_{|G|}).
\end{align*}
This completes the proof.
\end{proof}
\end{comment}

%In the next we analyze the product of element orders of groups with a normal cyclic Sylow subgroup.
The next lemma collects results on the product of element orders of groups with a normal cyclic Sylow subgroup, that follow from Lemma~2.4 and Lemma~2.6 in \cite{AK}.
\begin{lem}\label{lem:mercede}
Let $p$ be a prime and let $G$ be a finite group satisfying $G = P \rtimes F$, where $P$ is a cyclic Sylow $p$-subgroup and $(p, |F|) = 1$. Write $Z=C_F(P)$. Then 
\begin{itemize}
    \item [(i)] $\rho(G) = \rho(P)^{|Z|} \rho(F)^{|P|}.$
    \item [(ii)] $\rho(G) \mid \rho(P)^{|F|} \rho(F)^{|P|},$ with equality if and only if $C_F(P)=F$.
\end{itemize}
\end{lem}

When  $C_F(P) \neq F$ we have the following result.
\begin{cor}\label{cor:mercede}
Let $p$ be a prime and let $G$ be a finite group satisfying $G = P \rtimes F$, where $P$ is a cyclic Sylow $p$-subgroup and $(p, |F|) = 1$.  If $C_F(P) \neq F$, we have $$\rho(G) \leq q^{-q} \rho(C_{|G|}),$$ where $q$ is the smallest prime dividing $|G|$.
\end{cor}

\begin{proof}
Let $Z=C_F(P)$ and assume that $Z \neq F$. Applying Lemma~\ref{lem:mercede} to $G$ and \cite{GP}*{Theorem 3} to $F$, we have
\begin{align*}
    \rho(G) &= \rho(P)^{|Z|} \rho(F)^{|P|}\leq \rho(P)^{|Z|} \rho(C_{|F|})^{|P|} = \frac{1}{\rho(P)^{|F\setminus Z|}} \rho(C_{|G|}).
\end{align*}
Therefore we only need to check that $$q^q \leq \rho(P)^{|F\setminus Z|},$$ which is obviously true since $q \leq p$.
\end{proof}

It will be also useful to have information about $\rho(G)$  when $G$ has a non-cyclic normal Sylow subgroup.

\begin{pr}\label{nonciclico}
Let $p$ be an odd prime and let $G$ be a finite group satisfying $G = P \rtimes F$, where $P$ is a non-cyclic Sylow $p$-subgroup and $(p, |F|) = 1$. Then
$$\rho(G) \leq \left(\frac{|P|}{p}\right)^{|G|} \rho(F)^{|P|}.$$
In particular, $\rho(G) \leq q^{-q} \rho(C_{|G|}),$ where $q$ is the smallest prime dividing $|G|$.
\end{pr}

\begin{proof}
Since $P$ is not cyclic, $o(x) \leq \frac{|P|}{p}$ for every $x \in P$. Now, $G= \cup_{x \in F} Px$, hence if $x \in F$ and $y \in P$ we have $(xy)^{o(x)} \in P$ and $o(xy)$ divides $\frac{|P|}{p}o(x)$. It follows that
\[
\rho(G) = \prod_{x\in F}\prod_{y\in P} o(yx) \leq \prod_{x\in F}\prod_{y\in P} \frac{|P|}{p}o(x) = \prod_{y\in P} \left(\frac{|P|}{p}\right)^{|F|} \rho (F)\leq \left(\frac{|P|}{p}\right)^{|G|} \rho(F)^{|P|}.
\]

Finally, from Remark~\ref{rem:p} and  \cite{GP}*{Theorem 3} we have 
\[
\rho(G) \leq (\rho(C_{|P|})p^{-p})^{|F|} \rho(C_{|F|})^{|P|} = p^{-p |F|} \rho(C_{|G|}) \leq p^{-p} \rho(C_{|G|}) \leq q^{-q} \rho(C_{|G|}).
\]
\end{proof}

For a non-cyclic nilpotent group we have the following result.
\begin{pr}\label{pr: nilpotent}
Let $G$ be a non-cyclic nilpotent group of order $n$ and let $q$ be the smallest prime dividing $n$. Then 
\[
%\r(G) \leq  \frac{1}{q^{\frac{n}{q} (q-1)}}\r(C_{n}).
\r(G) \leq  q^{-\frac{n}{q} (q-1)}\r(C_{n}).
\]
\end{pr}
\begin{proof}
We proceed by induction on the number of prime divisors of $n$. Assume first that there exists $\alpha\geq 1$ such that $n=q^{\alpha}$. As $G$ is non-cyclic, we have $o(x) \leq q^{\alpha -1}$ for every $x \in G$. Let $M$ be a maximal subgroup of $G$. Applying \cite{GP}*{Theorem 3} we have
\begin{align*}
    \r(G) = \r(M)\prod_{x\in G\setminus M}o(x)\leq \r(C_{q^{\alpha-1}})(q^{\alpha-1})^{q^{\alpha}-q^{\alpha-1}}.
\end{align*}
Since $\r(C_{q^{\alpha}})=\r(C_{q^{\alpha-1}})q^{\alpha(q^{\alpha}-q^{\alpha-1})}$, it follows that
\begin{align*}
   % \r(G)\leq \r(C_{q^{\alpha}})\frac{1}{q^(q^{\alpha}-q^{\alpha-1})}=\frac{1}{q^{\frac{n}{q} (q-1)}}\r(C_{n}),
   \r(G)\leq \r(C_{q^{\alpha}})q^{-(q^{\alpha}-q^{\alpha-1})}=q^{-\frac{n}{q} (q-1)}\r(C_{n}),
\end{align*}
and in this case the result follows.

Assume now that at least two primes divide $n$.
Since $G$ is not cyclic, there exists a non-cyclic Sylow $p$-subgroup $P$ such that $G$ can be written as $G=P\times H$ with $(|P|,|H|)=1$. By induction and applying \cite{GP}*{Theorem 3} to $H$, we have
\begin{align*}
   % \rho(G) &= \rho(P)^{|H|} \rho(H)^{|P|}\leq \left(\frac{1}{p^{\frac{|P|}{p}(p-1)}}\r(C_{|P|})\right)^{|H|}\r(H)^{|P|}\\
    %&\leq \frac{1}{p^{\frac{|G|}{p}(p-1)}}\r(C_{|P|})^{|H|}\r(C_{|H|})^{|P|}=\frac{1}{p^{\frac{n}{p}(p-1)}}\r(C_{|G|}),
    \rho(G) &= \rho(P)^{|H|} \rho(H)^{|P|}\leq \left(p^{-\frac{|P|}{p}(p-1)}\r(C_{|P|})\right)^{|H|}\r(H)^{|P|}\\
    &\leq p^{-\frac{|G|}{p}(p-1)}\r(C_{|P|})^{|H|}\r(C_{|H|})^{|P|}=p^{-\frac{n}{p}(p-1)}\r(C_{|G|}),
\end{align*}
and the result follows from Remark~\ref{rem: q<p} as $q \leq p$.
\end{proof}

%%%%%%%%%%%%%% SECTION

\section{Proof of Theorem~\ref{mainthm} and some applications}\label{main}

In this section we prove the main result of the paper. We recall that, for $p$ a prime, a finite group $G$ is said to be \textit{$p$-nilpotent} if $G$ has a normal $p$-complement, i.e. there exists a normal subgroup $H$ and a Sylow $p$-subgroup $P$ of $G$ such that $HP = G$ and $H \cap P$ is trivial. Moreover, we say that a group $G$ is metacyclic
if its derived subgroup $G'$ is cyclic and the factor group $G/G'$ is cyclic.

\begin{proof}[Proof of Theorem \ref{mainthm}]
By Proposition~\ref{pr: nilpotent}, we can assume that $G$ is a supersoluble non-metacyclic group of order $n$ and we proceed by induction on $n$. If $p$ is the greatest prime dividing $n$, there exist a subgroup $H \leq G$ with $(|H|,p)=1$ such that $G=P \rtimes H$, where $P$ is the Sylow $p$-subgroup of $G$. Clearly $q$ divides $|H|$.

Firstly assume that $P$ is cyclic. Then we show that $H$ is not metacyclic. Arguing by contradiction, let $H'$ and $H/H'$ be cyclic. Since $P$ is cyclic, it follows that $G'$ centralizes $P$ and $G' \leq P \times H'$ is cyclic. Moreover, from $G/G' = PG'/G' \times HG'/G'$ we have $G/G'$ is cyclic as $H/H'$ so is. Therefore $G$ is metacyclic, which is a contradiction.

Therefore, when $P$ is cyclic we can apply Lemma~\ref{lem:mercede} and the induction argument on $H$ to conclude that the result holds.

Now assume that $P$ is not cyclic. If $G= P \times H$, from Lemma~\ref{rhocoprime} it follows that $\rho(G)=\rho(P)^{|H|} \rho(H)^{|P|}$. Then by Proposition \ref{pr: nilpotent} and Remark \ref{rem: q<p} the bound is obtained. 

Therefore we can assume that $C_{H}(P) < H$. 
If $H$ is not metacyclic, again we can apply Lemma \ref{lem:mercede}  and the induction argument on $H$. 

Hence suppose that both $P$ is a non-cyclic group and $H$ is a metacyclic group. 
In this case we have
\[
\rho(G) \leq \rho(H)^{|P|}\left(\frac{|P|}{p}\right)^{|G|} \leq \rho(C_{|H|})^{|P|}\left(\frac{|P|}{p}\right)^{|G|}.
\]
Thus we need to check if 
\[
%\rho(C_{|H|})^{|P|}\left(\frac{|P|}{p}\right)^{|G|} \leq \frac{1}{q^{\frac{|G|}{q} (q-1)}}\r(C_{|H|})^{|P|}\r(C_{|P|})^{|H|},
\rho(C_{|H|})^{|P|}\left(\frac{|P|}{p}\right)^{|G|} \leq q^{-\frac{|G|}{q} (q-1)}\r(C_{|H|})^{|P|}\r(C_{|P|})^{|H|},
\]
which is equivalent to prove that
\[
%\left(\frac{|P|}{p}\right)^{|G|} \leq \frac{1}{q^{\frac{|G|}{q} (q-1)}}\r(C_{|P|})^{|H|}.
\left(\frac{|P|}{p}\right)^{|G|} \leq q^{-\frac{|G|}{q} (q-1)}\r(C_{|P|})^{|H|}.
\]
Expanding all the values above, we have
\[
q^{\frac{p^{\alpha}(q-1)}{q}} \leq p^{\frac{p^{\alpha+1}-2p^{\alpha}+1}{p-1}},
\]
which reduces to prove that $p^
{\alpha}(p-q-1)+q \geq 0$, that is true as $p>q\geq 2$. This completes the proof.
\end{proof}

\begin{proof}[Proof of Corollary B]
Since $G$ has a Sylow tower, there exists a prime $p$ and a Sylow $p$-subgroup of $G$ such that $G=P \rtimes H$, where $H \leq G$ such that $(p, |H|)=1$. If $G= P \times H$, then $\rho(G)=\rho(P)^{|H|} \rho(H)^{|G|}$ and the result follows by induction on $|G|$. Therefore assume that $C_H(P) < H$. If $P$ is cyclic, the case follows from Corollary~\ref{cor:mercede}. Assume therefore that $P$ is not cyclic, then by Proposition~\ref{nonciclico} the bound is obtained and we are done.

\end{proof}

We finish the section dealing with groups of order $p^{\alpha}q^{\beta}$, where $p, q$ are primes with $p>q$.

 \begin{thm}\label{thm:pq}
Let $G$ be a non-cyclic group of order $n=p^{\alpha}q^{\beta}$, where $p > q$ are primes. Then
\[
\rho(G) \leq q^{-q} \rho(C_n).
\]
\end{thm}

\begin{proof}
Assume by way of contradiction that $\rho(G) > q^{-q} \rho(C_n)$. Firstly we show that there exists $x \in G$ such that $o(x) > p^{\alpha -1}q^{\beta -1}$. Assume that $o(x) \leq p^{\alpha -1}q^{\beta -1}$ for every $x \in G$. Then we have
\[
\rho(G) \leq (p^{\alpha -1}q^{\beta -1})^{p^{\alpha} q^{\beta}}=((p^{\alpha -1})^{p^{\alpha}})^{q^{\beta}} ((q^{\beta -1})^{q^{\beta}})^{p^{\alpha}}.
\]
From Remark~\ref{rem:p} it follows that
\[
\rho(G) \leq (\rho(C_{p^{\alpha}})p^{-1})^{q^{\beta}} (\rho(C_{q^{\beta}})q^{-1})^{p^{\alpha}} \leq \rho(C_n) p^{-q^{\beta}}q^{-p^{\alpha}} \leq \rho(C_n) q^{-q},
\]
which yields a contradiction. 
Thus there exists $x \in G$ such that $o(x) > p^{\alpha-1}q^{\beta-1}$ with $|G: \langle x \rangle| < pq$. We distinguish two cases. If $p$ divides $|G: \langle x \rangle|$, then the only possibility is that $|G: \langle x \rangle|=p$. As a consequence, $G$ has cyclic Sylow $q$-subgroups, and so $G$ is $q$-nilpotent by \cite[10.1.9]{robinson}. Therefore $G$ admits a Sylow tower, and the result follows from Corollary~B.
Suppose then that $p$ does not divide  $|G: \langle x \rangle|$. In this case $|G: \langle x \rangle|=q^{\gamma}$, where $q^{\gamma-1}<p$ as $q^{\gamma}<pq$.
It follows that there exists a Sylow $p$-subgroup $P$ of $G$ such that $P \leq \langle x \rangle$. Then $\langle x \rangle \leq N_G(P)$ and  $|G : N_G(P)|$ divides $q^{\gamma}$. However, $|G : N_G(P)| = 1 + kp$, for $k \geq 0$. If $k = 0$, then $P$ is normal in $G$ and the result follows from Corollary \ref{cor:mercede}. If $k > 0$, since $|G : N_G(P)| > p$ and $q^{\gamma-1}<p$, we have $|G : N_G(P)| =q^{\gamma}$ and $N_G(P) = \langle x \rangle$. Therefore,  $P \leq Z(N_G(P))$ and $G$ is $p$-nilpotent by \cite[10.1.8]{robinson}
In this case $G = Q \rtimes P$. Therefore $G$ has a Sylow tower and applying again Corollary~B we are done. This completes the proof.
\end{proof}

\section{Product of element orders of a Frobenius group}

In the following, we estimate the product of element orders of a Frobenius group. We recall that a finite group $G$ is said to be a  \textit{Frobenius group} if $G$ has a subgroup $H$ such that $H \cap H^x=1$ for all $x \in G \setminus H$.  Frobenius proved that if $G$ is such a group, then
\[
N= G \setminus \bigcup_{x\in G} (H \setminus \{1\})^x
\]
is a normal subgroup of $G$, and $G=NH$ with $N \cap H=1$. In this case $H$ is called a \textit{Frobenius complement} and $N$ the \textit{Frobenius kernel}. 
As a consequence, in a Frobenius group $G=NH$, we have $(|N|,|H|)=1$.

\begin{pr}\label{frobenius}
Let $G$ be a Frobenius group with Frobenius kernel $N$ and Frobenius complement $H$. Then
\[
 \rho(G)=\rho(N)\rho(H)^{|N|}.
\]
In particular, if $G$ has order $n$ and $q$ is the smallest prime dividing $n$, we have
\[
\rho(G) \leq q^{-q}\rho(C_n).
\]
\end{pr}
\begin{proof}
As $G$ is a Frobenius group, it can be covered by its Frobenius kernel and all its Frobenius complements. Therefore
\[
\rho(G)=\rho(N)\rho(H)^{|N|}
\]
as $H$ has $|G:H|$ distinct conjugates in $G$. It follows that
$$
\rho(G)= \rho(N)\rho(H)^{|N|} \leq \rho\left(C_{|N|}\right)\rho\left(C_{|H|}\right)^{|N|}=\frac{\rho(C_n)}{\rho\left(C_{|N|}\right)^{|H|-1}}.
$$
Hence it suffices to show that 
$$
q^q \leq \rho\left(C_{|N|}\right)^{|H|-1}.
$$
Since $(|N|,|H|)=1$, we distinguish cases: $q$ divides $|N|$, or $q$ divides $|H|$. Assume the former holds. Then we have
$$
q \leq |N| \leq \rho(N) \leq \rho(C_{|N|}).
$$
Since $q$ is the smallest prime dividing $|G|$, we have $q<|H|$. In other words, $q \leq |H|-1$, and the result follows.

Assume now that $q$ divides $|H|$. Then 
\begin{equation}\label{case2}
q < \rho(N) \leq \rho(C_{|N|}).  
\end{equation}
If $|H|>q$, then $|H|-1 \geq q$ and the inequality follows from Equation \eqref{case2}. If $|H|=q$, then every prime $p$ dividing $|N|$ is greater than $q$, and so $\rho(N)\geq p^2 >q^2$. Since $|H|-1=q-1$, we have
$$
\rho(C_{|N|})^{|H|-1}\geq \rho(N)^{q-1} \geq q^{2(q-1)} \geq q^q.
$$
This completes the proof.
\end{proof}

We conclude this section with an example, in which we compute the product of element orders of a group obtained as direct product of a Frobenius group and a cyclic group of coprime order.

\begin{example}
Let $n>5$, and let $G=F \times C$ be the direct product of a Frobenius group $F$, with the cyclic kernel $N$ and a cyclic complement $H$, and a cyclic group $C$ such that $(|F|, |C|)=1$ and $|F||C|=n$. Then
\[
\rho(G)=\frac{\rho(C_n)}{\rho(C_{|N|})^{|C|(|H|-1)}}.
\]
Indeed, by Lemma~\ref{rhocoprime}, and Proposition~\ref{frobenius}, we have
\begin{equation*}
\rho(G) =\rho(F)^{|C|}\rho(C)^{|F|}=\rho(N)^{|C|}\rho(H)^{|N||C|}\rho(C)^{|F|}.   
\end{equation*}
Since $(|F|, |C|)=1$, applying Lemma~\ref{rhocoprime} it follows that 
\begin{align*}
\rho(G) &=\rho(N)^{|C|} \rho(H)^{|N||C|}\rho(C)^{|F|} \\[2mm]
        &=\frac{\rho(N)^{|C||H|} \rho(H)^{|N||C|}\rho(C)^{|F|}}{\rho(N)^{|C|(|H|-1)}}=\frac{\rho(C_{|F|})^{|C|} \rho(C)^{|F|}}{\rho(N)^{|C|(|H|-1)}} \\[2mm]
        &=\frac{\rho(C_n)}{\rho(N)^{|C|(|H|-1)}},
\end{align*}
and we are done.
\end{example}

\subsection*{Acknowledgments}
The authors are grateful to professors G. A. Fern\'andez-Alcober, P. Longobardi, and M. Maj for interesting conversations. They also thank the anonymous referees for giving us valuable comments.

\subsection*{Data Availability Statement}
\noindent This manuscript has no associated data.

\end{document}